\documentclass[12pt]{amsart}
\usepackage{amsmath,amssymb,amsbsy,amsfonts,amsthm,latexsym,amsopn,amstext,
                            amsxtra,euscript,amscd}
\usepackage{url}

\newtheorem{theorem}{Theorem}
\newtheorem{lemma}[theorem]{Lemma}

\newtheorem{cor}[theorem]{Corollary}

\def\cA{{\mathcal A}}
\def\cB{{\mathcal B}}

\def\cG{{\mathcal G}}

\def\cK{{\mathcal K}}

\def\cR{{\mathcal R}}

\def\cV{{\mathcal V}}

\def\cX{{\mathcal X}}

\def\Z{\mathbb{Z}}

\def\({\left(}
\def\){\right)}
\def\<{\langle}
\def\>{\rangle}

\def\fM{{\mathfrak M}}
\def\fE{{\mathfrak E}}

\def\fl#1{\left\lfloor#1\right\rfloor}

\def\mand{\qquad\mbox{and}\qquad}

\newcommand{\twolinesum}[2]{\sum_{\substack{{\scriptstyle #1}\\
{\scriptstyle #2}}}}

\begin{document}

\title[Products of Small Integers in Residue Classes]
{Products of Small Integers in Residue Classes and Additive Properties 
of Fermat Quotients}

\author{Glyn Harman} 
\address{Department of Mathematics, Royal Holloway, University of London, 
Egham, Surrey TW20 0EX, UK}
\email{g.harman@rhul.ac.uk}

\author{Igor E.~Shparlinski} 
\address{Department of Pure Mathematics, University of 
New South Wales, Sydney, NSW 2052, Australia}
\email{igor.shparlinski@unsw.edu.au}

\subjclass{11G07, 11T06, 11Y16}
\keywords{Short products, residue classes, character sums, sieve}

\begin{abstract} We show that for any $\varepsilon > 0$ and  a sufficiently large 
cube-free $q$,  any reduced residue class modulo $q$ can be represented as a
product of $14$ integers from the interval $[1, q^{1/4e^{1/2} + \varepsilon}]$.
The length of the interval is at the lower limit of what is possible before 
the Burgess bound on the smallest quadratic nonresidue is improved. 
We also consider several variations of this result and give applications to 
Fermat quotients. 
\end{abstract}

\maketitle

\section{Introduction}

As usual, we say that an integer $n$ is $y$-smooth if all 
prime divisors $p\mid n$ satisfy $p \le y$. We write 
$$\beta = \frac14e^{-1/2}.
$$  
By a result of 
Harman~\cite[Theorem~3]{Harm} for any $\varepsilon > 0$ and  a sufficiently large 
cube-free $q$, every reduced residue class modulo $q$ contains
a $q^{\beta+ \varepsilon}$-smooth positive integer $s \le q^{9/4+\varepsilon}$. 
Clearly this result is the best possible (in terms of $\beta$) until at least 
the Burgess bound~\cite{Burg1,Burg2}
on the smallest quadratic nonresidue is improved.
Harman~\cite[Theorem~3]{Harm}  also gives similar, albeit weaker, results for 
non cube-free moduli $q$.

Here we are mostly interested in the number of small factors of $n$ rather than 
in its size. 
More precisely our goal is   to minimize the values of $k$ such that 
for any  $\varepsilon > 0$ and  a sufficiently large 
cube-free $q$, for any integer $a$ with $\gcd(a,q)=1$,  
there is always a solution to the congruence 
\begin{equation}
\label{eq:Cong}
n_1\ldots n_k \equiv a \pmod q, \qquad 1 \le n_1, \ldots, n_k  \le q^{\beta+ \varepsilon}.
\end{equation}

We remark that $\beta= 0.1516\ldots$ and it is certainly the limit 
of what one may hope to obtain without improving the Burgess bound~\cite{Burg1}
on the smallest quadratic non-residue. For large intervals, several 
results in this direction have been obtained by Garaev~\cite{Gar2}.
For example, Garaev~\cite{Gar2} notices that for any $\varepsilon > 0$ and  
a sufficiently large cube-free $q$ every  $a$ with $\gcd(a,q)=1$
can be represented modulo $q$ as a product of $k=8$ 
positive integers up to $q^{1/4 + \varepsilon}$ (which is an immediate 
consequence of~\cite[Theorem~2]{Gar2}).  It is a feature of all current methods that if our variables are of size $q^{\vartheta}$
then we require $k \vartheta > 2$.  In that sense both our result
with $k = 14$ and
that of Garaev~\cite{Gar2} are best possible
at present (note that $13 \beta < 2$).
Several related questions, also involving multiplicative subgroups 
of the unit group $\Z_q^*$ of the residue ring modulo $q$, have been studied
by Cilleruelo and Garaev~\cite{CillGar}.

 Note that although formally~\cite[Theorem~3]{Harm} does not give
any upper bound on the number of factors in a $q^{\beta+ \varepsilon}$-smooth 
positive integer $s \le q^{9/4+\varepsilon}$ with $s \equiv a \pmod q$ such a 
bound can easily be derived via simple combinatorial arguments. 
More precisely, one  combines
together prime divisors of $n$ in a greedy way into factors of size at most
$q^{\beta+ \varepsilon}$.   The argument of~\cite{Harm}
is flexible enough to impose additional restrictions on the 
prime factors of the integers $s$ to solve~\eqref{eq:Cong} with $k=18$ and with more work that can be reduced to $k=16$. 
We can do a little better, however, by combining this approach with the ideas
of Balog~\cite{Bal}  and Garaev~\cite{Gar2} to derive the following result. 

\begin{theorem}
\label{thm:Cong} For any $\varepsilon > 0$ and  a sufficiently large 
cube-free $q$, for any integer $a$ with $\gcd(a,q)=1$,  
there is always a solution to the congruence~\eqref{eq:Cong}
with $k= 14$. 
\end{theorem}

Some of our motivation to investigate the solvability of~\eqref{eq:Cong}
for small values of $k$ comes from studying the additive
properties of the
{\it Fermat quotient\/} $q_p(u)$  modulo a prime $p$, which is defined as the unique integer
with 
$$
q_p(u) \equiv \frac{u^{p-1} -1}{p} \pmod p, \qquad 0 \le q_p(u) \le p-1.
$$
We also define 
$$
q_p(kp) = 0, \qquad k \in \Z.
$$
Clearly the function $q_p(u)$ is periodic with 
period $p^2$. 
For any integers $r$, $u$ and $v$ with $\gcd(uv,p) = 1$
we have
\begin{equation}
\label{eq:add struct1}
q_p(u) + q_p(v)\equiv q_p(uv)  \pmod p
\end{equation}
and
\begin{equation}
\label{eq:add struct2}
q_p(u+rp) \equiv q_p(u) - ru^{-1} \pmod p, 
\end{equation}
see, for example,~\cite[Equations~(2) and~(3)]{ErnMet}.

Fermat quotients appear  in various questions of
computational and algebraic number theory, see the survey~\cite{ErnMet}
of classical results and also~\cite{BFKS,Ihara,Shkr,Shp5} for results 
about vanishing Fermat quotients,~\cite{ChWi1,ChWi2,HaSh,OstShp,Shp3,Shp4}
for results about the distribution, fixed points and value set
and~\cite{Chang,GomWint,Shp1,Shp2,Shp3} for bounds of exponential and multiplicative
character sums. 

Furthermore, Chen and Winterhof~\cite{ChWi2} have recently studied additive properties 
of Fermat quotients and their generalisations. In particular, Chen and Winterhof~\cite{ChWi2} 
study the question 
of solvability of the congruence 
\begin{equation}
\label{eq:Cong F}
q_p(u_1) + \ldots + q_p(u_k) \equiv a \pmod p, \qquad 1 \le u_1, \ldots, u_k  \le U,
\end{equation}
for some fixed integer $k$ and a sufficiently large parameter $U$ (and also a
congruence with a generalisation of Fermat quotients). Clearly the method of~\cite{ChWi2},
based on bounds of exponential sums has a natural limit of $U \ge p^{1/2+\varepsilon}$
for an arbitrary small $\varepsilon > 0$ coming from the non-triviallity 
range of the Burgess bound, see~\cite{Burg1,Burg2} and also~\cite[Theorem~12.6]{IwKow}
for a modern treatment. 
Here, we observe that Theorem~\ref{thm:Cong}  applied with $q = p^2$ and 
combined with~\eqref{eq:add struct1}
and~\eqref{eq:add struct2} allows us to study~\eqref{eq:Cong F}
for much smaller values of $U$.  Indeed, 
it follows from~\eqref{eq:add struct2} that for any integer $b$ with $\gcd(b,p)=1$,
there exists an integer $a$ with $\gcd(a,p)=1$ such that $q_p(a) \equiv b \pmod p$.
Hence, we derive from Theorem~\ref{thm:Cong} that for any $\varepsilon > 0$, a sufficiently large 
prime $p$ and $U \ge p^{1/2e^{1/2} + \varepsilon}$, for any integer $a$ with $\gcd(a,p)=1$,  
there is always a solution to the congruence~\eqref{eq:Cong F}
with $k= 14$. 


In actual fact it is more efficient to analyse the problem of Fermat quotients 
more closely and establish a variant of 
Theorem~\ref{thm:Cong}  for the congruence
\begin{equation}
\label{eq:Cong G}
n_1\ldots n_k \equiv a u \pmod q, \qquad 1 \le n_1, \ldots, n_k  \le q^{\beta}, \ u \in \cG,
\end{equation}
with a multiplicative subgroup $\cG$ of $\Z_q^*$.  This follows since
$q_p(u) = 0$ if $u=r^p$ for some $r$ with $\gcd(r,p)=1$.
So instead of solving
\eqref{eq:Cong}  we now have much more flexibility
and solve~\eqref{eq:Cong G} with $q=p^2$ and where $\cG$ is the group of $p$-th 
powers $(\bmod{\, q })$.  We note that $\cG$ has order $p-1 \gg q^{1/2}$.  This motivates the following result.

\begin{theorem}\label{thm:Cong G}
 For any $\varepsilon > 0$ and  a sufficiently large 
cube-free $q$, and  a multiplicative subgroup $\cG$ of $\Z_q^*$  
of order $t \gg q^{1/2}$ for any integer $a$ with $\gcd(a,q)=1$,  
there is always a solution to the congruence~\eqref{eq:Cong G} with
$ k=9$.
\end{theorem}

In particular, Theorem~\ref{thm:Cong G} implies:

\begin{cor}\label{cor:Cong F}
 Let $\varepsilon > 0$.  Suppose  $p$ is a sufficiently large 
prime and $U \ge p^{1/2e^{1/2} + \varepsilon}$.  Then, for any integer $a$ with $\gcd(a,p)=1$,  
there is always a solution to the congruence~\eqref{eq:Cong F} with $k=9$.
\end{cor}

%

Our approach can also be used to study 
the solvability of~\eqref{eq:Cong} for almost all reduced residue 
classes $a \pmod q$ and obtain several more results 
complementing those of  Cilleruelo and Garaev~\cite{CillGar}
and Garaev~\cite{Gar2}. We state one such result as follows.

\begin{theorem}
\label{thm:Cong aa} For any $\varepsilon > 0$ and  a sufficiently large 
cube-free $q$, for all but $o(q)$ integers $a \in \{0, \ldots, q-1\}$ with $\gcd(a,q)=1$,  
there is always a solution to the congruence~\eqref{eq:Cong}
with $k= 7$.
\end{theorem}

In particular, Theorem~\ref{thm:Cong aa} implies:

\begin{cor}
\label{cor:Cong F aa} For any $\varepsilon > 0$, a sufficiently large 
prime $p$ and $U \ge p^{1/2e^{1/2} + \varepsilon}$, for all but $o(p)$ integers 
$a \in \{0, \ldots, p-1\}$ with  $\gcd(a,p)=1$,  
there is always a solution to the congruence~\eqref{eq:Cong F}
with $k=  7$. 
\end{cor}

\section{Preparations}

\subsection{Notation}

Throughout the paper, any implied constants in the symbols $O$,
$\ll$  and $\gg$  may depend on the real parameter 
$\varepsilon > 0$. We recall that the
notations $U = O(V)$, $U \ll V$ and  $V \gg U$ are all equivalent to the
statement that the inequality $|U| \le c V$ holds with some
constant $c> 0$.

We define the constants
 $\psi = 2^{1/36}$,  $\xi = \psi-1$,
and for a real $A$ and an integer $a$,  
write $a \sim A$ to indicate 
$a \in [A,\psi A]$.   We also write $\rho = e^{-1/2}$.

We use $\Z_q^*$ to denote the unit group of the residue ring modulo $q$.

As usual, we write $\varphi(n)$ for the Euler function and 
$\tau(n)$ to represent the number of positive integer
divisors of an integer $n \ge 1$ for which we
recall the following well-known estimates
\begin{equation}
\label{eq:phitau}
 \tau(q) = q^{o(1)}
 \mand  q \ge  \varphi(q) \gg \frac{q}{\log \log q}
\end{equation}
as $q \to \infty$,
see~\cite[Theorems~317 and~328]{HaWr}.


In the following $\kappa$ always denotes the ratio
$$
\kappa = \frac{\varphi(q)}{q}.
$$

\subsection{Some basic results}

The next result is a well-known elementary consequence of the identity
$$
\sum_{d \mid \gcd(n,q)} \mu(d) = \begin{cases} 1 &\text{if $\gcd(n,q)=1$}\\ 0 &\text{otherwise.}
\end{cases}
$$

\begin{lemma}\label{LA}
For any $M \ge 1, q \ge 2$ we have
$$
\twolinesum{m\sim M}{\gcd(m,q)=1} 1 = \xi M \kappa + O(\tau(q))\,.
$$
\end{lemma}

When $M$ is small in relation to $q$ it is useful to have the following result.

\begin{lemma}\label{LB}
For $q \ge 2, M > \log q$ we have
$$
\twolinesum{m\sim M}{\gcd(m,q)=1} 1 =O(M \kappa)\,.
$$
Here the implied constant is absolute.
\end{lemma}
\begin{proof}
This follows from \cite[Theorem 2.2]{HR}.
\end{proof}

Combining the above results enables us to establish a result  which is needed in Subsection~\ref{sec:Balog}. 

\begin{lemma}\label{LC}
For $N \ge q^{1/4} > 1, 0 < \zeta < 1$ we have
$$
\twolinesum{N^{\zeta} \le p \le N}{\gcd(p,q)=1}\twolinesum{m\sim N/p}{\gcd(m,q)=1} 1 =   (\xi  \log(1/\zeta) + o(1))   \kappa N \,.
$$
\end{lemma}

\begin{proof}
From Lemmas~\ref{LA} and~\ref{LB} together with a trivial bound we have
$$
\twolinesum{m\sim N/p}{\gcd(m,q)=1} 1 = \xi N \kappa p^{-1} + O(\vartheta_p)
$$
where
\begin{equation}
\label{eq:3Ranges}
\vartheta_p = \begin{cases} \tau(q) &\text{if $p< N/\tau(q)$,}\\ 
\kappa N/p &\text{if $N/\tau(q) \le p \le N/\log q$,}\\
N/p &\text{if $p > N/\log q$.}
\end{cases}
\end{equation}
(note the ranges may partially overlap and the second range may be empty). 
By the Mertens formula, see~\cite[Equation~(2.15)]{IwKow}, for any real $Y > X \ge 2$ we have
$$
\sum_{X \le p \le Y}\frac{1}{p} = \log \frac{\log Y}{\log X}  +O\(\frac{1}{\log X}\).
$$
Hence 
$$
\twolinesum{N^{\zeta} \le p \le N}{\gcd(p,q)=1} \xi N \kappa p^{-1}  = 
 (\xi  \log(1/\zeta) + o(1))   \kappa N
 $$
 gives us the main term (where we have also noted that there are only $O(1)$ primes $p \mid q$ with $p > N^{\zeta}$).

For the error term we consider the 3 possible ranges in~\eqref{eq:3Ranges}
separately. 

For the first range,  by Prime Number Theorem and the bound
$$
\kappa  \gg \frac{1}{\log \log q},
$$
see~\eqref{eq:phitau}, we derive 
$$
\sum_{p < N/\tau(q)}  \vartheta_p \ll  \sum_{p < N/\tau(q)} \tau(q) = (1+o(1)) \frac{N}{\log N} = o(\kappa N).
$$
For the second range (provided it is not empty) using the above Mertens formula and~\eqref{eq:phitau}, 
we obtain
\begin{equation*}
\begin{split}
\sum_{N/\tau(q) \le p \le N/\log q}   \vartheta_p &\ll  
\kappa N \sum_{N/\tau(q) \le p \le N/\log q} \frac{1}{p}\\
 &= \kappa N \(  \log \frac{\log N - \log \log q}{\log N - \log \tau(q)}  +O\(\frac{1}{\log N}\)\)\\
& =  \kappa N \(  \log \frac{\log N + o(\log N)}{\log N +o(\log N)}  +O\(\frac{1}{\log N}\)\) \\
& =  \kappa N \(  \log (1  +o(1)) +O\(\frac{1}{\log N}\)\) 
= o(\kappa N).
\end{split}
\end{equation*}
Finally, for the third range, similarly, we have 
\begin{equation*}
\begin{split}
\sum_{N/\log q< p \le N}   \vartheta_p &\ll  
 N \sum_{N/\log q  < p \le N} \frac{1}{p}\\
 &= N \(  \log \frac{\log N}{\log N - \log \log q}  +O\(\frac{1}{\log N}\)\)\\
& =   N \(  \log \frac{\log N}{\log N +O(\log \log N)}  +O\(\frac{1}{\log N}\)\) \\ 
& \ll  N \frac{\log \log  N}{\log N} = o(\kappa N).
\end{split}
\end{equation*}
The desired result now follows.
\end{proof}

\subsection{Using a simple idea of Balog}
\label{sec:Balog}

Instead of establishing a variant of~\cite[Lemma~1]{Harm} which uses an idea of Friedlander~\cite{Friedl} to obtain a lower bound of the correct order of magnitude for the integers we wish to count, we return to the original idea of Balog~\cite{Bal}.
That is, we count products of two numbers $mn$, and note, for
any set  $\cA$,
$$
\twolinesum{mn \in \cA}{p|mn \Rightarrow p<x^{\alpha}} 1
\ \ \ge
\twolinesum{mn \in \cA}{p|m \Rightarrow p<x^{\alpha}} 1
\ - \ 
\twolinesum{mn \in \cA}{\exists p|n,  p>x^{\alpha}} 1.
$$
We do this for simplicity as it would take considerable effort to obtain the correct order lower bound in view of the complicated structure we impose on the numbers we
have eventually to count. We write for convenience  $\cB= \{n:~\gcd(n,q)=1\}$. Our main auxiliary result is then as follows.

\begin{lemma}
\label{lem:Harm} 
Suppose $R > N > q^{1/4}$.  Let $\varepsilon >0$ be given and a sequence $b_r$ supported on the interval $[R, \psi^{34}R]$.
Suppose that $\cA \subseteq \cB$ is a set such that  for some
$\lambda > 0$ and $\eta = \varepsilon^3$,
\begin{equation}
\label{eq:Cond Sum}
\twolinesum{rnm \in \cA}{ m, n \sim N} a_n b_r
=\lambda
\twolinesum{rnm \in \cB}{ m, n \sim N}  a_n b_r
+
O(\lambda x^{1 - \eta})
\end{equation}
for any sequence $a_n = O(1)$. 
Write $\zeta = \rho(1 + \varepsilon)$. Let
$$
c_n = \begin{cases}1 &\text{if $p\mid n \Rightarrow p < N^{\zeta}$}\\
0 &\text{otherwise.}
\end{cases}
$$
Then
\begin{equation}\label{mainaux}
\twolinesum{rmn \in \cA}{m, n \sim N} b_r c_n c_m \ge (2 + o(1)) \lambda \log(1+\varepsilon)
(\kappa \xi N)^2 \sum_{r \in \cB} b_r +
O(\lambda x^{1 - \eta})\,.
\end{equation}
\end{lemma}

\begin{proof}  
Using  the  observation of Balog~\cite{Bal}, we have
$$
\twolinesum{rmn \in \cA}{m, n \sim N} b_r c_n c_m \ge E - F
$$
where
$$
E = \twolinesum{rmn \in \cA}{m, n \sim N} b_r  c_m \,, \qquad F = \twolinesum{rmn \in \cA}{m, n \sim N} b_r h_n\,,
$$
and $h_n = 1-c_n$.  By~\eqref{eq:Cond Sum}
$$
E = \lambda \twolinesum{rmn \in \cB}{m, n \sim N} b_r c_m +  O(\lambda x^{1 - \eta}).
$$
Now
$$
\twolinesum{rmn \in \cB}{m, n \sim N} b_r c_m = \twolinesum{m \in \cB}{m \sim N}  c_m  \ 
\twolinesum{n \in \cB}{n \sim N} 1  \ \sum_{r \in \cB} b_r  \,.
$$
Lemmas~\ref{LA} and~\ref{LC} then give 
$$
\twolinesum{n \in \cB}{n \sim N} 1 =(1 + o(1))  \kappa \xi N , \qquad
\twolinesum{m \in \cB}{m \sim N}  c_m = (1 + \log \zeta + o(1)) \kappa \xi N  \,,
$$
where we have noted (since $\rho > \frac12$) that
$$
\twolinesum{m \in \cB}{m \sim N}  c_m =\twolinesum{m \in \cB}{m \sim N}1 - \sum_{N^{\zeta} < p\le N} 
\twolinesum{m \in \cB}{m \sim N/p} 1\,.
$$
Thus 
$$
E = \lambda(1+ \log \zeta  + o(1))
(\kappa \xi N)^2 \sum_{r \in \cB} b_r +  O(\lambda x^{1 - \eta}) \,.
$$
Similarly
$$
F = \lambda(- \log \zeta  + o(1))
(\kappa \xi N)^2 \sum_{r \in \cB} b_r +  O(\lambda x^{1 - \eta})\,.
$$
Since $1 + 2 \log \zeta = 2 \log(1+\varepsilon)$ we obtain~\eqref{mainaux}.
\end{proof}

Now we define the {\it multiset\/}
\begin{equation}\label{defk}
\cK = \{k=mn~:~ m, n \sim N, p \mid mn \Rightarrow p < N^{\zeta}\}\,,
\end{equation}
where the integers $k$ are counted with multiplicity.
For a real $x>1$ and integers $a$ and $q$ with $\gcd(a,q)=1$ 
and a subgroup $\cG$ of $\Z_q^*$ 
we define by $\cA_{a,q}(\cG; x)$ the set of integers $s \in [x,2x]$ 
with $s \equiv a u \pmod q$ for some $u \in \cG$. 
We record a special case of  Lemma~\ref{lem:Harm}
that applies to the set  $\cA =  \cA_{a,q}(\cG; x)$. 

\begin{cor}
\label{cor:AP-Subgr} Assume that the conditions of Lemma~\ref{lem:Harm} 
holds with $x=N^2 R > x_0(\varepsilon)$ for the set $\cA = \cA_{a,q}(\cG; x)$
with $\lambda = t/ \varphi(q)$, where $t = \# \cG$ and $x_0(\varepsilon)$ depends
only on $\varepsilon$ and is sufficiently large. 
Then 
$$ 
\sum_{\substack{rk \in \cA_{a,q}(\cG; x)\\ k\in \cK}} b_r 
 \ge \varepsilon \frac{t\kappa ^2 \xi^2 N^2}{\varphi(q)}
 \sum_{r \in \cB}
b_r + O\(t q^{-1} x^{1-\eta}\). 
$$
\end{cor}

In particular, for the extreme case $\cG = \{1\}$,  
we write $\cA_{a,q}(x)$ for $\cA_{a,q}(\{1\},x)$
and obtain:

\begin{cor}
\label{cor:AP} Assume that the conditions of Lemma~\ref{lem:Harm} 
holds with $x=N^2 R > x_0(\varepsilon)$ for the set $\cA =  \cA_{a,q}(x)$
with $\lambda = 1/ \varphi(q)$,  where  $x_0(\varepsilon)$ depends
only on $\varepsilon$ and is sufficiently large. 
Then  
$$
\sum_{\substack{rk \in \cA_{a,q}(x)\\ k\in \cK}} b_r 
 \ge \varepsilon  \frac{\kappa^2 \xi^2 N^2}{\varphi(q)}
 \sum_{r\in \cB}
b_r + O\(q^{-1} x^{1-\eta}\). 
$$
\end{cor}

\subsection{Character sums}

Let $\cX$ be the set of all $\varphi(q)$ multiplicative
characters modulo $q$ and let $\cX^*$ be the set 
of nonprincipal characters $\chi \ne \chi_0$.
We now recall the Burgess bound for sums of multiplicative
characters modulo cube-free integers which we
present in the following simplified form, see~\cite[Theorems~12.5 and~12.6]{IwKow}.

\begin{lemma}
\label{lem:Burg} There is an absolute constant $c> 0$ 
such that for any fixed $\delta \in (0,1/2)$, a cube-free integer $q$ and 
an arbitrary integer $M \ge q^{1/4 + \delta}$, 
for any $\chi \in \cX^*$ we have
$$
\left|\sum_{m\le M} \chi(m)\right| \ll M^{1-c\delta^2}.
$$
\end{lemma}

Finally, we need the following simple bound  which follows from 
the orthogonality of characters
and   which we  refer to as the {\it mean-value estimate for character sums\/}.

\begin{lemma} 
\label{lem:Aver}
For $N \ge 1$ and any sequence of complex numbers $a_n$ we have
$$
\sum_{\chi \in \cX} \left| \sum_{n \le N} a_n \chi(n) \right|^2 \le \varphi(q) (N/q + 1) \sum_{n \le N} |a_n|^2.
$$
\end{lemma}

\subsection{Products in arithmetic progressions}

We now define 
\begin{equation}
\label{eq:delta alpha}
\delta = 1/200 \mand \alpha = \(\tfrac14 + \varepsilon\) (2+\delta+ 2 \varepsilon)^{-1}.
\end{equation}
For a given $q$, we consider the set of integers $r$ that are 
products of 33 primes of the form 
\begin{equation}
\label{eq:set r}
r = \ell_1\ldots \ell_{21} p_1\ldots p_{8} s_1 \ldots s_4\mand
\gcd(r,q)=1, 
\end{equation}
where 
\begin{equation}
\label{eq:primes}
\ell_1, \ldots, \ell_{21} \sim q^{\delta}, \quad 
p_1, \ldots, p_8 \sim q^{3/20}, \quad  s_1, s_2, s_3, s_4 \sim q^{1/20},
\end{equation}
and let $b_r$ be the characteristic 
function of this set. We note that $b_r$ is supported on the interval $[R, \psi^{33} R]$ with $R=q^{3/2 + \delta}$.

We now show that for any sufficiently small $\varepsilon> 0$ 
the conditions of Lemma~\ref{lem:Harm} 
are satisfied for this choice of $b_r$ with $N=q^{1/4 + \varepsilon} = x^{\alpha}$ upon writing $x=N^2R$.  

\begin{lemma}  
\label{lem:Cong} Let $\varepsilon > 0$ be sufficiently small, $q> 1$ and $N=q^{1/4 + \varepsilon}$.    
Suppose that the sequence  $b_r$ is the characteristic 
function of the set  defined by~\eqref{eq:set r} and~\eqref{eq:primes}.
Then for integers $a$ and $q$ with $\gcd(a,q)=1$ and such that $q$ is 
cube-free  we have
$$
\sum_{\substack{rmn \in \cA_{a,q}(x)\\ m, n \sim N}}
a_n b_r =  \frac{1}{\varphi(q)}\sum_{\substack{rmn \in \cB\\ m, n \sim N}}
a_n b_r + O\(q^{-1} x^{1-\eta}\)
$$
with $\eta = \varepsilon^3$,
 $R=q^{3/2 + \delta}$ and $x=N^2R$, and any sequence $a_n$ 
 satisfying $|a_n| \le n^{o(1)}$. 
\end{lemma}

\begin{proof} 
We start with the observation that  if $b_r \ne 0$ and $m, n \sim N$ 
then due to the choice of our parameters we always have 
$$
rmn \in [N^2R, \psi^{33}N^2 R] \subset [x,2x].
$$
In particular, if $b_r \ne 0$ and $m, n \sim N$  then the condition 
$rmn \in  \cA_{a,q}(x)$ is equivalent to the 
congruence $rmn \equiv a \pmod q$ and the condition 
$rmn \in  \cB$ is merely equivalent to $\gcd(mn,q)=1$. 

 Using the orthogonality 
of characters we write 
$$
\sum_{\substack{rmn \in \cA_{a,q}(x)\\ m, n \sim N}}
a_n b_r = \sum_{\substack{rmn \in \cB\\ m, n \sim N}}
a_n b_r \frac{1}{\varphi(q)} \sum_{\chi\in \cX} \chi(rmna^{-1}).
$$
Changing the order of summation, we obtain the asymptotic 
formula 
\begin{equation}
\label{eq:M and E}
\sum_{\substack{rmn \in \cA_{a,q}(x)\\ m, n \sim N}}
a_n b_r = \fM + O\(\fE\), 
\end{equation}
where the main term
\begin{equation}
\label{eq:M}
\fM = \frac{1}{\varphi(q)} \sum_{\substack{rmn \in \cB\\ m, n \sim N}}
a_n b_r 
\end{equation}
comes from  the contribution 
of the principal character $\chi_0$ and the error term is given by
$$
\fE= \frac{1}{\varphi(q)} \sum_{\chi\in \cX^*} \left|
\sum_{\substack{rmn \in \cB\\ m, n \sim N}}
a_n b_r \chi(rmn)\right|
=   \frac{1}{\varphi(q)} \sum_{\chi\in \cX^*} \left|
\sum_{\substack{r  \in \cR\\ m, n \sim N}}
a_n b_r \chi(rmn)\right| , 
$$
where $\cR$ is the set of $r$   defined by~\eqref{eq:set r} and~\eqref{eq:primes}
(note that due to the presence of  characters the condition $rmn \in \cB$
can now be dropped). 

Hence
\begin{equation}
\label{eq:fE1}
\fE= \frac{1}{\varphi(q)} \sum_{\chi\in \cX^*} \left|\sum_{ m \sim N}\chi(m) \right| 
 \left|
\sum_{r  \in \cR} \sum_{ n \sim N} 
a_n b_r \chi(rn) \right| . 
\end{equation}

We now  use the argument deployed in~\cite{Gar2}, we however put it 
in a different form which optimally extracts all available information 
about the character sums involved (thus in case the bound on error 
terms is important it leads to stronger estimates). This approach also 
seems to be more direct and since it may have some other applications,
we present it  in full detail.

For a real $\omega> 0$ we consider the character sums over 
primes
$$
V_\omega(\chi)  = \sum_{\ell\sim q^{\omega}} \chi(\ell). 
$$
which we use with $\omega =3/20$ and $\omega=\delta$. 
We also consider 
the weighted sums
$$
W(\chi) = 
\sum_{m \sim N} \sum_{n \sim N}  \sum_{v \in \cV} a_n  \chi(mnv), 
$$
where $v$ runs through  the set $\cV$ of $q^{1/2 + o(1)}$ products 
$v=p_7 p_8 s_1 s_2 s_3 s_4$
over all $p_7,p_8,s_1.s_2,s_3,s_4$ as in~\eqref{eq:set r}. 
Recalling the definition of $b_r$, we write~\eqref{eq:fE1} as
\begin{equation}
\label{eq:fE2}
\fE= \frac{1}{\varphi(q)} \sum_{\chi\in \cX^*}\left|V_{3/20}(\chi)\right|^6  \left|V_{\delta}(\chi)\right|^{21} \left|W(\chi) \right|. 
\end{equation}
We now note that the
currently available information 
about the sums $V_{3/20}(\chi)$, $V_{\delta}(\chi)$, 
 and  $W(\chi)$ consists of the inequality
\begin{equation}
\label{eq:indiv}
\max_{\chi \in \cX^*} \left|W(\chi)\right|  
\le N^{1+o(1)} q^{1/2} \max_{\chi \in \cX^*} \left|\sum_{m \sim N}   \chi(m)\right| \ll  N^{2-c_0\varepsilon^2}q^{1/2}
\end{equation}
with some absolute constant $c_0> 0$ for all sufficiently small $\varepsilon>0$
that follows from  Lemma~\ref{lem:Burg} 
and also the inequalities
\begin{equation}
\label{eq:aver1}
 \sum_{\chi \in \cX} \left|V_{3/20}(\chi) \right|^{12}  \left|V_\delta(\chi) \right|^{40} \ll q^{2}, 
\qquad \sum_{\chi \in \cX} \left|V_\delta(\chi) \right|^{400} \ll q^{2},
\end{equation}
and 
\begin{equation}
\begin{split}
\label{eq:aver2}
 \sum_{\chi\in \cX} & \left|\sum_{m \sim N} \sum_{n \sim N} a_n   \sum_{v \in \cV} \chi(mnv)\right|^2\\
& \qquad \qquad\qquad  \ll N^2 q^{3/2+o(1)}\(1 + N^2q^{-1/2}\) , 
\end{split}
\end{equation}
implied by  Lemma~\ref{lem:Aver}.  Since for the above choice of parameters we have $N^2 > q^{1/2}$, 
the inequality~\eqref{eq:aver2} simplifies as
\begin{equation}
\label{eq:aver3}
 \sum_{\chi\in \cX}  \left|\sum_{m \sim N} \sum_{n \sim N} a_n   \sum_{v \in \cV} \chi(mnv)\right|^2
  \ll N^4 q^{1+o(1)}. 
\end{equation}

 We now write $ \left|W(\chi) \right| = \left|W(\chi) \right|^{199/200}  \left|W(\chi) \right|^{1/200}
$ and apply~\eqref{eq:indiv}, deriving from~\eqref{eq:fE2}
\begin{equation}
\label{eq:fE3}
\begin{split}
\fE
\le \frac{1}{\varphi(q)} \(N^{2-c_0\varepsilon^2}q^{1/2}\)^{1/200} \sum_{\chi\in \cX^*}&\left|V_{3/20}(\chi)\right|^6 \\
& \quad \left|V_{\delta}(\chi)\right|^{21} \left|W(\chi) \right|^{199/200}. 
\end{split}
\end{equation}
Finally, since 
$$
\frac{1}{2} + \frac{1}{400} + \frac{1}{400/199} = 1
$$
 by the H{\"o}lder inequality, applied to the sum in~\eqref{eq:fE3}, and extending the summation to 
 all $\chi \in \cX$, we obtain 
\begin{equation*}
\begin{split}
\fE\le \frac{1}{\varphi(q)} & \(N^{2-c_0\varepsilon^2}q^{1/2}\)^{1/200} 
\( \sum_{\chi \in \cX} \left|V_{3/20}(\chi) \right|^{12}  \left|V_\delta(\chi) \right|^{40}\right)^{1/2} \\
& \qquad \(\sum_{\chi\in \cX} \left|V_{\delta}(\chi)\right|^{400}\)^{1/400}
\(\sum_{\chi\in \cX} \left|W(\chi) \right|^2\)^{199/400}. 
\end{split}
\end{equation*}
Recalling~\eqref{eq:aver1} and~\eqref{eq:aver3},  we derive
\begin{equation}
\label{eq:fE4}
\begin{split}
\fE& \le \frac{1}{\varphi(q)}   \(N^{2-c_0\varepsilon^2}q^{1/2}\)^{1/200}  q^{1+1/200} \(N^4 q^{1+o(1)}\)^{199/400}\\
&\le N^2 q^{1/2+\delta + o(1)} N^{-c_0\varepsilon^2/200} =  xq^{-1}N^{-c_0\varepsilon^2/200+o(1)}  .
\end{split}
\end{equation}
The proof is completed by 
 combining~\eqref{eq:M} and~\eqref{eq:fE4} with~\eqref{eq:M and E}.
\end{proof}

\subsection{Products  in subgroups}

Before embarking on the proof of Theorem~\ref{thm:Cong G} we also 
require one additional result,  that gives an upper bound on the 
number of solutions to the congruence
\begin{equation}\label{eq:extra}
xu \equiv y \pmod{q} \qquad 1\le x,y \le X,\ u \in \cG,
\end{equation}
with a multiplicative subgroup $\cG$ of  $\Z_q^*$,
which is given in~\cite[Corollary~7.9]{KoSh}. 
We note that in~\cite{KoSh} only the case of a prime modulus $q=p$ 
is considered, but it is easy to check that the argument works
for any integer $q\ge 1$.

\begin{lemma}\label{lem:subgroup}
Given a multiplicative subgroup $\cG$ of  $\Z_q^*$ with order $t$ satisfying 
$t \gg q^{1/3}$ and an integer $X \ge q^{3/4}t^{-1/4}$, 
the number of solutions to the congruence~\eqref{eq:extra}
is at most $X^2 t q^{-1 + o(1)}$.
\end{lemma}

We also recall that several more bounds on the number of solutions to~\eqref{eq:extra}
are given in~\cite[Theorem~1]{BKS}.

We replace~\eqref{eq:delta alpha} to define $\delta$ and $\alpha$   now with
$$
\delta = 1/200 \mand \alpha =  (\tfrac14 + \varepsilon) (\tfrac54+\delta+ 2 \varepsilon)^{-1}.
$$

For a given $q$, we consider the set of integers $r$ that are 
products of 28 primes of the 
form 
\begin{equation}
\label{eq:set r G}
r = \ell_1 \ldots \ell_{21} p_1 p_2  p_3 s_1 s_2 s_3 s_4
\mand \gcd(r,q)=1, 
\end{equation}
where 
\begin{equation}
\label{eq:primes G}
\ell_1, \ldots, \ell_{21} \sim q^{\delta}, \quad 
p_1, p_2, p_3 \sim q^{3/20}, \quad  s_1, s_2, s_3, s_4 \sim q^{1/20}, 
\end{equation}
and let $b_r$ be the characteristic 
function of this set. We remark that we need $r$ to be expressible as two factors of size about $q^{3/8}$
as well as having the right combinatorial properties for our final argument.

\begin{lemma}  
\label{lem:Cong G} Let $\varepsilon > 0$ be sufficiently small, $q> 1$ and $N=q^{1/4 + \varepsilon}$.
Suppose that the sequence  $b_r$ is the characteristic 
function of the set  defined by~\eqref{eq:set r G} and~\eqref{eq:primes G}.
Then for integers $a$ and $q$ with $\gcd(a,q)=1$ and such that $q$ is 
cube-free and a subgroup $\cG \subseteq \Z_q^*$ of order $t \gg q^{1/2}$ we have
$$
\sum_{\substack{rmn \in \cA_{a,q}(\cG; x)\\ m, n \sim N}}
a_n b_r =  \frac{t}{\varphi(q)}\sum_{\substack{rmn \in \cB\\ m, n \sim N}}
a_n b_r + O\(tq^{-1} x^{1-\eta}\)
$$
with $\eta = \varepsilon^3$, $R=q^{3/4 + \delta}$ and $x=N^2R$, and any sequence $a_n$ 
 satisfying $|a_n| \le n^{o(1)}$. 
\end{lemma}

\begin{proof}  
We proceed as in the proof of Lemma~\ref{lem:Cong}.
We note that we  count each desired solution $t$ times  by considering
$$
mnr \equiv auv \pmod{q}, 
$$
where 
$$
r\ \text{as in}~\eqref{eq:primes G}, \qquad m\sim N, \qquad n \sim N, 
 \qquad u,v \in \cG\,.
$$
As before in the proof of Lemma~\ref{lem:Cong} we use multiplicative characters 
to obtain a main term
\begin{equation}
\label{eq:M G}
\fM = \frac{t^2}{\varphi(q)} \sum_{\substack{rmn \in \cB\\ m, n \sim N}}
a_n b_r
\end{equation}
for the corresponding sum, which we write as 
$$
\sum_{\substack{mnr \equiv auv \pmod{q}\\ m, n \sim N\\ u,v \in \cG}} a_n b_r
= \fM + O(\fE).
$$
We also write 
\begin{equation*}
\begin{split}
V_\delta(\chi) &= \sum_{\ell \sim q^{\delta}} \chi(\ell),\\
W_1(\chi) &= 
\sum_{m \sim N}  \sum_{p_1,p_2 } \sum_{s_1}\sum_{\ell_1, \ldots ,\ell_5} \sum_{u\in \cG} \chi(mp_1p_2s_1 \ell_1 \ldots \ell_5\overline{u}), \\
W_2(\chi) &= 
 \sum_{n \sim N}   \sum_{p_3} \sum_{s_2,s_3,s_4} \sum_{\ell_6, \ldots,\ell_{20}}  \sum_{v\in \cG} a_n 
\chi(np_3s_2s_3s_4\ell_6\ldots \ell_{20}\overline{v})\,.
\end{split}
\end{equation*}
So the expression for the error term $\fE$ corresponding 
to~\eqref{eq:fE1} is now
$$
\fE= \frac{1}{\varphi(q)} \sum_{\chi\in \cX^*} \left|V_\delta(\chi)\right| \left|W_1(\chi) \right|  \left|W_2(\chi) \right| . 
$$
Working in a similar manner to previously we estimate this as
$$
\fE \le \max_{\chi\in \cX^*}  \left|W_1(\chi) \right|^{\delta}  \(\frac{1}{\varphi(q)}\sum_{\chi\in \cX} \left|V_{\delta}(\chi)\right|^{400}\)^{1/400}
S_1^{199/400} S_2^{1/2},
$$
where
$$
S_j = \frac{1}{\varphi(q)} \sum_{\chi\in \cX} \left|W_j(\chi)\right|^2.
$$
From  Lemma~\ref{lem:Burg}, we have, for ${\chi\in \cX^*}$,
$$
\left|W_1(\chi)\right| \le N^{1-c_0\varepsilon^2} t Z,
$$
where $Z$ is the maximum number of all admissible 
products of the form $p_1p_2s_1\ell_1\ldots\ell_5$ satisfying~\eqref{eq:primes G}, 
so 
\begin{equation}
\label{eq:Z}
Z \ll q^{3/10 + 1/20+ 1/40} =    q^{3/8}. 
\end{equation}
This and~\eqref{eq:aver1}
imply the bound
\begin{equation}
\label{eq:W1 G}
\fE \le \(N^{1-c_0\varepsilon^2} t  q^{3/8}\)^{\delta} q^{\delta/2}
S_1^{199/400} S_2^{1/2}. 
\end{equation}

Also, we can estimate $S_1$ as $S_1\le R q^{o(1)}$ where $R$ is the number of solutions to
the congruence
$$
m_1p_1p_2s_1\ell_1\ldots\ell_5 u \equiv m_2p_3p_4s_2 \ell_6 \ldots \ell_{10}v \pmod{q}, 
$$
where 
$$p_j, s_j, \ell_j\ \text{are as in}~\eqref{eq:primes G}, \qquad m_1,m_2 \sim N, 
 \qquad u,v \in \cG\,.
$$
Now $R \ll t Q$ where $Q$ is the number of solutions to~\eqref{eq:extra}   with 
$$
X \ll NZ \ll X,
$$ 
where $Z$ is given by~\eqref{eq:Z}. 
As $t \ge q^{1/2}$ and $N>q^{1/4}$, we have
$$
NZ \gg  N q^{3/8} \ge q^{5/8} \ge q^{3/4} t^{-1/4}.
$$
Hence Lemma~\ref{lem:subgroup} applies and 
we obtain 
$$
Q  \le (NZ)^2 t  q^{-1+o(1)} = N^2 t q^{-1/4+o(1)}. 
$$

We also proceed similarly for $S_2$.
Hence
\begin{equation}
\label{eq:S1 S2 G}
S_1\le N^2 t^2 q^{-1/4+o(1)}\mand 
S_2\le N^2 t^2 q^{-1/4+o(1)}.  
\end{equation}

We can now substitute the  estimates~\eqref{eq:S1 S2 G} in~\eqref{eq:W1 G}, to deduce that
\begin{equation*}
\begin{split}
\fE &\le 
 \(N^{1-c_0\varepsilon^2} t  q^{3/8}\)^{\delta} q^{\delta/2}\(N^2 t^2 q^{-1/4+o(1)}\)^{(1-\delta)/2} \(N^2 t^2 q^{-1/4+o(1)}\)^{1/2} \\
&\le 
N^2  t^2 q^{-1/4  + \delta + o(1)}N^{-c_0\delta\varepsilon^2}\,.
\end{split}
\end{equation*}
Hence
\begin{equation}
\label{eq:fE G}
\fE\le xt^2q^{-1}M^{-c_0\varepsilon^2/200+o(1)} .
\end{equation}
The proof is completed by combining~\eqref{eq:M G} and~\eqref{eq:fE G} upon recalling that we are counting each solution $t$ times.
\end{proof}

\section{Proofs of Main Results}

\subsection{Proof of Theorem~\ref{thm:Cong}}

We fix some sufficiently small $\varepsilon > 0$. 
Let $\alpha$ and $\delta$ be as in~\eqref{eq:delta alpha} and let $\eta$ be as in 
Corollary~\ref{cor:AP}. 
We also choose $x$ as in  Lemma~\ref{lem:Cong}.  We remark that $N^{\zeta} <x^{\beta + \varepsilon}$.

For integers $a$ and $q$ with $\gcd(a,q)=1$ and such that $q$ is 
cube-free we consider 
the number $T$ of solutions to the congruence 
\begin{equation}
\label{eq:rk cong}
rk\equiv a \pmod q
\end{equation}
where $r$ is  defined by~\eqref{eq:set r} and~\eqref{eq:primes} and
$k\in \cK$, where the multiset $\cK$ is defined by~\eqref{defk}.

Combining Corollary~\ref{cor:AP} and Lemma~\ref{lem:Cong}, we see 
that 
\begin{equation}
\label{eq:T}
T = 
\sum_{\substack{rk \in \cA_{a,q}(x)\\ k\in \cK}} b_r 
 \ge \varepsilon  \frac{\kappa^2 \xi^2 N^2}{\varphi(q)}
 \sum_{r\in \cB}
b_r + O\(q^{-1} x^{1-\eta}\). 
\end{equation}

By the prime number theorem there are $q^{3/2+ \delta+o(1)}$ values of $r$ given 
by~\eqref{eq:set r} and~\eqref{eq:primes} and for each of 
them $q^{3/2 + \delta} \ll r \ll q^{3/2 + \delta}$.
Hence, for a sufficiently small $\varepsilon > 0$, after simple calculations,
we obtain 
$$
T \ge  xq^{-1+o(1)}.
$$ 
In particular, $T>0$. Let $(k,r)$ be one of the solutions to~\eqref{eq:rk cong}.
Clearly $r$ has $8$ prime factors of size
$q^{3/20} <q^{\beta+ \varepsilon}$. We 
return to the other $25$ factors after an initial discussion of $m$ and $n$
showing that they are both
 products of at most 3 integer
factors of size at most 
$u = q^{\beta+ \varepsilon}$. 
Indeed, let $\widetilde{p}_1\ge \ldots \ge \widetilde{p}_\nu$ be prime divisors of $m$.
Define $h$ by the condition 
$$
\widetilde{p}_1\ldots \widetilde{p}_h \le u < \widetilde{p}_1\ldots \widetilde{p}_h \widetilde{p}_{h+1}.
$$
Then for 
$$
v_1 = \widetilde{p}_1\ldots \widetilde{p}_h, \qquad v_2 = \widetilde{p}_{h+1}, \qquad v_3 = \frac{n}{v_1v_2}
$$
we obviously have $v_1v_2 v_3 = m$ and also 
$$
\max\{v_1,v_2, v_3\} \le \max\{u, u, m/u\} \le u,
$$
provided that $\varepsilon> 0$ is sufficiently small.
In particular, in what follows, we always assume that 
$$
\varepsilon < \frac12 \delta.
$$ 

  Now, clearly $\min\{v_1,v_2, v_3\} \le m^{1/3} < q^{(1+\delta)/12}$.
For convenience, suppose that $v_1 \ge v_2 \ge v_3$. 
So, if $v_2>q^{1/10}$ then $v_3< q^{1/20 + \delta/4}$.  Hence we can combine 
$s_1$ and $s_2$ with $v_2$ and $v_3$
 to produce new variables not exceeding $q^{3/20+ \delta/4}$.  So we have written
 $ms_1s_2 = g_1g_2g_3$ 
 say with each positive integer $g_j \le q^{\beta+ \varepsilon}$.  However,
$$
g_1 g_2 g_3 \ell_1 \ldots \ell_{11} \le q^{2/5 +5 \delta/4},
$$
So, suppose $g_j \le q^{3/20 - \delta}$.  We can include variables $\ell_1,\ldots \ell_h$ so that
$$
q^{3/20 - \delta} \le g_j \ell_1 \ldots \ell_h \le q^{3/20}.
$$
We can do this for each $g_j \le q^{3/20 - \delta}$ and since, for 
a sufficiently large $q$, we have
\begin{equation*}
\begin{split}
g_1 g_2 g_3 \ell_1 \ldots \ell_{11}& \le q^{2/5 + 5\delta/4} 
< \psi^{-3} q^{3 \times 3/20 - 3 \delta}\\
& \le q^{3 \times 3/20} \(\max \{\ell \sim q^{\delta}\}\)^{-3}
\end{split}
\end{equation*} 
we use up all  $\ell_1,\ldots, \ell_{11}$.
In this way 
$ms_1s_2 \ell_1 \ldots \ell_{11}$ 
has been expressed as the product of three variables not exceeding 
$q^{\beta+ \varepsilon}$.

The same argument also applies to $n$, although in this case we need only use $10$ of the $\ell_j$ variables.
We have thus reduced our product to $14$ variables as desired.  

\subsection{Proof of Theorem~\ref{thm:Cong G}}

We now proceed as in the proof of Theorem~\ref{thm:Cong}
by using  Lemma~\ref{lem:Cong G} instead of Lemma~\ref{lem:Cong} and applying Corollary~\ref{cor:AP-Subgr}. We have $3$ variables of the correct shape immediately in
$p_1,p_2,p_3$.  We can use the same argument as before to reduce $\ell_1\ldots \ell_{21}s_1 \ldots s_4 mn$ to a product of $6$ variables not exceeding $q^{\beta + \varepsilon}$.  This gives the $9$ variables as required.

\subsection{Proof of Theorem~\ref{thm:Cong aa}}

Suppose that $\cR$ is the set of multiplicative inverses $(\bmod{\,q})$ of the exceptional set of $a$.
All we need do is prove that for any set $\cR$ with $|\cR|=q^{1+o(1)}$ there is a solution to
$$
r n_1 \ldots n_7 \equiv 1 \pmod{q}, \quad 1 \le n_1, \ldots, n_7  \le q^{\beta+ \varepsilon}, \ r\in \cR.
$$
To modify the proof of  Theorem~\ref{thm:Cong} we keep the definiton of $\alpha, \delta$ 
from~\eqref{eq:delta alpha} and 
we initially  solve
\begin{equation*}
\begin{split}
rmnp s_1s_2s_3s_4 &\ell_1   \ldots \ell_{31} \equiv 1 \pmod{q}, \\ 
\ell_1, \ldots, \ell_{31} \sim q^{\delta}  \qquad p &\sim q^{3/20}, \qquad s_1,s_2,s_3,s_4 \sim q^{1/20} .
\end{split}
\end{equation*}
As before
 $mns_1s_2s_3s_4 = g_1 \ldots g_6$ 
with each $g_j \le q^{\beta+ \varepsilon}$.  Now
\begin{equation*}
\begin{split}
g_1 \ldots g_6 \ell_1 \ldots \ell_{31} & \le q^{17/20 + 5\delta/4} <  \psi^{-6} q^{6 \times 3/20 - 6 \delta}\\
& \le q^{6 \times 3/20} \(\max \{\ell \sim q^{\delta}\}\)^{-6}\,.
\end{split}
\end{equation*} 
We can therefore combine some of the $\ell_j$ variables with each $g_k$ in turn to obtain $6$  new variables
not exceeding $g_j \le q^{\beta+ \varepsilon}$.  We thus end up with $7$ variables of the required form as desired.

\section{Additional Results}

A natural question is to see how far short our results fall 
from what would be known assuming the Generalised Riemann Hypothesis.
Under that assumption we quickly obtain the well-known conditional bound for a short sum over a non-principal multiplicative character $\chi$ modulo $q$:
\begin{equation}
\label{eq:GRH Bound}
\sum_{m \le M} \chi(m) \ll M^{1/2}q^{o(1)} ,
\end{equation}
as $q\to \infty$, see~\cite[Section~1]{MoVau}; it can also be 
 derived   from~\cite[Theorem~2]{GrSo}.

We also obtain  the following
conditional extension of Theorem~\ref{thm:Cong G}
without any need to use Lemma~\ref{lem:Harm}, where
as usual we use $\fl{x}$ to denote 
the integer part of real $x$. 

\begin{theorem}
\label{thm:Conggrh} Assume the Generalised Riemann Hypothesis.  For any $\beta \in (0,1)$ and  a sufficiently large 
$q$, a multiplicative subgroup $\cG$ of $\Z_q^*$  
of order $t = q^\vartheta$, for any integer $a$ with $\gcd(a,q)=1$,  
there is always a solution to the congruence~\eqref{eq:Cong G} 
with $k= \fl{2(1-\vartheta)/\beta} + 1$.   
\end{theorem}

\begin{proof}
Since $\cG$ is a group, for fixed $u$ the number of solutions to $vw=u$ with  $v,w \in \cG$ is $t$.  Hence the 
number of solutions to~\eqref{eq:Cong G}  is $t^{-1} S$ where $S$ is the number of solutions to
$$
n_1\ldots n_k \equiv a vw \pmod q, \qquad 1 \le n_1, \ldots, n_k  \le q^{\beta}, \ v,w \in \cG.
$$
Using character sums we obtain $S= M +O(E)$ where
$$
\frac{M}{t^2} =
 \frac{1}{\varphi(q)}\left(\sum_{n \le q^{\beta}} \chi_0(n)  \right)^k 
= q^{k\beta -k} \varphi(q)^{k-1} + O\left( q^{(k-1)\beta - 1 + o(1)} \right),
$$
and
$$
E = \frac{1}{\varphi(q)} \sum_{\chi\ne \chi_0} \left|\sum_{n \le q^{\beta}} \chi(n)\right|^k \left|\sum_{u\in \cG} \chi(u)  \right|^2.
$$
Using~\eqref{eq:GRH Bound} applied to the sum over $n$ and the mean value theorem for character sums for the sum of the sums over $u$
gives
$$
E \le q^{k \beta/2 + o(1)} t.
$$
We thus obtain that the number of solutions to~\eqref{eq:Cong G}  is
$$
S/t = tq^{k\beta-k} \varphi(q)^{k-1} + O\(tq^{(k-1)\beta - 1 +o(1)} + q^{k \beta/2 + o(1)}\). 
$$
Hence $S > 0$ for $q$ sufficiently large when
$$
k\beta-1 >  (k-1)\beta - 1 \mand \vartheta + k\beta - 1 > k \beta/2 .
$$
The first inequality is always satisfied as $\beta >0$, 
analysing the second inequality we obtain  the result. 
\end{proof}

Comparing this (taking $\cG$  to be the trivial subgroup) with our unconditional result we see that for $\beta = 1/4e^{1/2}$ we
make no saving on the number of variables. However,   we can reduce $\beta$ to $1/7 + \varepsilon$ and still only require $14$ variables.
The real benefit, of course, is that we can take any $\beta >0$ to obtain factors essentially as small as we wish. The major
constraint imposed by this approach is that the product of the variables must be of 
size at least $q^{2+\varepsilon}$.

Now let $\cG$ be the group of $p$th powers modulo $p^2$ as in the proof of Theorem~\ref{thm:Cong G}.  We
immediately deduce the following result which saves $2$ variables on 
Theorem~\ref{thm:Cong G} for $\beta = 1/4e^{1/2}$ .  

\begin{cor}
\label{cor:Cong GRH} Assume the Generalised Riemann Hypothesis.  
 For any $\beta \in (0,1/2)$  and a sufficiently large 
prime $p$  for any integer $a$ with $\gcd(a,p)=1$  and $U \ge p^{2\beta}$,
there is always a solution to the congruence~\eqref{eq:Cong F}
with $k= \fl{1/\beta} + 1$. 
\end{cor}

The next natural question is: what happens for almost all moduli? 
The auxiliary results used in proving the Bombieri-Vinogradov Theorem
(see~\cite[Chapter~28]{Dav})
immediately show that Theorem~\ref{thm:Conggrh} is true for $\cG=\{1\}$
unconditionally for all $q \in [Q,2Q]$ with $o(Q)$ exceptions
as well as for almost all prime $p = q\in [Q,2Q]$ with $o(Q/\log Q)$ exceptions. One can obtain results for non-trivial $\cG$, but
the results become complicated and do not have the full strength of Theorem~\ref{thm:Conggrh}.

An alternative approach to results for almost all $q$ is via a bound of 
Garaev~\cite{Gar1} of character sums for almost all moduli (which can be 
used in place of Lemma~\ref{lem:Burg}). This approach may lead to stronger 
results for some group sizes. 
 

Next, we would like a result for almost all 
$p^2$ in order to obtain the appropriate version of
Corollary~\ref{cor:Cong F} for  almost all prime squares.  This is possible
since  Matom\"aki~\cite{KM} has obtained an analogous version of the Bombieri-Vinogradov theorem for prime-squared moduli.
Using the Type~II sum estimates in~\cite[Section~3]{KM}, we are able quickly to deduce the following.

\begin{theorem}
\label{thm:CongKM} For any $\beta \in (0,1/2)$ and  a sufficiently large 
$Q$, for all but $o(Q^{1/2}/\log Q)$ exceptional prime squares $p^2\in [Q,2Q]$, for any integer $a$ with $\gcd(a,p)=1$,  
there is a solution to the congruence
$$
n_1\ldots n_k \equiv a \pmod {p^2}, \qquad 1 \le n_1, \ldots, n_k  \le Q^{\beta}, 
$$
with $k= \fl{2/\beta} + 1$.  
\end{theorem}

\begin{cor}
\label{cor:Cong Final} For any $\beta\in (0,1/2)$ and  a sufficiently large 
$T$, for all but $o(T/\log T)$ exceptional primes $p\in [T,2T]$,
and $U \ge p^{2\beta}$, for any integer $b$ with $\gcd(b,p)=1$,  
there is always a solution to the congruence~\eqref{eq:Cong F}
with $k= \fl{2/\beta} + 1$.   
\end{cor}

Of course, for $\beta = 1/4e^{1/2}$ this is worse than our Theorem~\ref{thm:Cong G} which is true for all $p$, but Corollary~\ref{cor:Cong Final} does hold for {\it all} $\beta > 0$.

Finally, we mention that a version of Theorem~\ref{thm:Cong G},
which involves multiplicative subgroups $\cG$ of $\Z_{p^2}^*$ of certain sizes
is possible via a version of a result of Garaev~\cite{Gar1} for almost all
prime squares, of the type given in~\cite[Theorem~8]{Shp3}.

\section{Comments}

We note that the choice of parameters~\eqref{eq:set r G} and~\eqref{eq:primes G}
is optimised for subgroups of order $t = q^{1/2+o(1)}$. The chief reason for this is
that our main application to Fermat quotients corresponds to subgroups of this size.
However, one can easily obtain a series of other results of the type of  Theorem~\ref{thm:Cong G}
for subgroups of other sizes. Furthermore, for other choices of parameters several
other versions of Lemma~\ref{lem:subgroup} may be of use. For example, one can
use~\cite[Theorem~1]{BKS} with other values of $\nu$ and also a similar 
estimate from~\cite{CillGar}. Furthermore, for some ranges of $q$, $t$ and $X$, 
one can obtain better estimates via bounds of multiplicative character sums.

\section*{Acknowledgement}

During the preparation  I.~E.~Shparlinski was supported in part
by ARC grant DP130100237.

\end{document}